\documentclass{amsart}
\usepackage[utf8]{inputenc}
\usepackage{amsmath}
\usepackage{amssymb}
\usepackage{amsthm}
\usepackage{xcolor}
\usepackage{tikz}
\usepackage{tikz-cd}
\newtheorem{lemma}{Lemma}[section]

\newtheorem{proposition}[lemma]{Proposition}
\newtheorem{hypothesis}[lemma]{Hypothesis}

\newtheorem{claim}[lemma]{Claim}

\theoremstyle{definition}

\usepackage[backend=bibtex]{biblatex}
\title[Correction]{Correction to: ``A spectral sequence for stratified spaces and configuration spaces of points''}
\author{Nir Gadish and Dan Petersen}
\begin{document}

\maketitle

The goal of this note is to correct some oversights in the paper \cite{spectralsequencestratification} by the second named author. The issue with Lemma 4.14 explained below was noticed by the the first author in the spring of 2018, and the workaround described in this note was worked out jointly during a visit of the first author to Stockholm in the spring of 2019. 

Let us briefly recall the setting of Section 4 of \cite{spectralsequencestratification}: $M$ denotes a fixed ambient space, and for any finite set $S$ there is defined an open subspace $F_{\mathcal A}(M,S)$ of ``$\mathcal A$-avoiding'' configurations inside the cartesian product $M^S$. For a sheaf $F$ on $M^S$ there is an explicit complex of sheaves $L^\bullet_F$ on $M^S$, such that $L^\bullet_F$ is quasi-isomorphic to the restriction of $F$ to $F_{\mathcal A}(M,S)$, extended by zero. When $F$ is the constant sheaf this complex of sheaves is denoted $L^\bullet(S)$. The goal is to prove representation stability for the Borel--Moore homology of $F_{\mathcal A}(M,S)$ under mild assumptions on $M$ and $\mathcal A$, by analyzing the combinatorics of $L^\bullet(S)$. This combinatorial structure is governed by the collection of posets $P_{\mathcal A}(S)$ (where $P_{\mathcal A}(S)$ is the poset of strata for the natural stratification of $M^S$, with the open stratum $F_{\mathcal A}(M,S)$ being the minimal element), and the natural map $P_{\mathcal A}(S) \times P_{\mathcal A}(T) \hookrightarrow P_{\mathcal A}(S \sqcup T)$ which is induced by the fact that the product of two closed strata is again a closed stratum. 

\section{The product in the twisted commutative algebra}\label{subsec: product}

The first oversight in the paper we want to discuss is rather minor, but worth making explicit. At the top of p.\ 2548 it is claimed that there is a natural map $L^\bullet(S) \boxtimes L^\bullet(T) \to L^\bullet(S \sqcup T)$. But in fact the only natural map goes in the other direction.  This error is fortunately cancelled by an equal and opposite error in the following sentence: a map in this direction is exactly what is needed to make $S \mapsto H^{-\bullet}(M^S,\mathbb D L^\bullet(S))$ a twisted commutative algebra, contrary to what the paper claims. Let us make explicit what this map of complexes of sheaves looks like. 

Note first that the space $M^{S \sqcup T} \cong M^S \times M^T$ has two natural stratifications: one with poset of strata $P_{\mathcal A}(S \sqcup T)$, and one with poset of strata $P_{\mathcal A}(S) \times P_{\mathcal A}(T)$. The former stratification is a refinement of the latter one. Now the functoriality of the construction $L^\bullet$ with respect to such refinements gives a map
$ L^\bullet(S \sqcup T) \to \widetilde L^\bullet,$
where %the quotient complex
$\widetilde L^\bullet$ is the resolution of the constant sheaf on $F_{\mathcal A}(M,S) \times F_{\mathcal A}(M,T)$ extended by zero to $M^S \times M^T$. The complex $\widetilde L^\bullet$ is a sum over chains in the product poset $P_{\mathcal A}(S) \times P_{\mathcal A}(T)$, with a differential given by \emph{inserting} elements into a chain, i.e.\ the linear dual of the usual homological differential defined by omitting an element of a chain. Thus the linear dual of the Eilenberg--Zilber shuffle map defines a map $\widetilde L^\bullet \to L^\bullet(S) \boxtimes L^\bullet(T)$. %Note that the Eilenberg--Zilber map is symmetric monoidal, which will be useful in our construction of a chain level twisted commutative algebra structure. 
Taking Verdier duals gives product maps \[ \mathbb D L^\bullet(S) \boxtimes \mathbb D L^\bullet(T) \rightarrow \mathbb D L^\bullet(S \sqcup T),\]
which are the ones we will be most concerned with.

\section{A subposet which is not an order ideal}
 The main subject of this note is an incorrect claim at the bottom of p.\ 2547, that the natural injection 
\begin{equation}\label{eq:poset_product}
 P_{\mathcal A}(S) \times P_{\mathcal A}(T) \hookrightarrow P_{\mathcal A}(S \sqcup T)\end{equation}
identifies the left hand side with an order ideal (i.e.\ a downwards closed subset) in the poset on the right hand side. Now although the left hand side of (1) is an order ideal in many naturally occurring examples, it is not true in general, and we do not see a non-tautological hypothesis that one could add to make this true in general. 

The fact that \eqref{eq:poset_product} is not necessarily an order ideal invalidates the proof of Theorem 4.15, which is the main result of Section 4 of the paper. Indeed, the proof of Theorem 4.15 relies on Lemma 4.14, which uses the claim that \eqref{eq:poset_product} is an order ideal, specifically in its first line claiming that one needs only keep track of indecomposable elements. This implicitly assumes that the induced multiplication
\[
\bigoplus_{i+j=n} H_{i-2}(\hat{0},\beta) \otimes H_{j-2}(\hat{0},\beta') \rightarrow H_{n-2}(\hat{0},\beta\times \beta')
\]
is surjective, as would follow from \eqref{eq:poset_product} being an order ideal: that would already imply that the product $[\hat{0},\beta]\times [\hat{0},\beta'] \rightarrow [\hat{0},\beta\times \beta']$ is an isomorphism of posets.

However, one can run the proof of Theorem 4.15 under a weaker hypothesis than \eqref{eq:poset_product} being an order ideal, as we will now explain. The key observation is that the only part of product structure that comes into the proof of Theorem 4.15 is multiplication by the trivial element $\hat{0}\in P_{\mathcal{A}}(1)$, where $1$ denotes a one-element set. We will require an additional hypothesis concerning the arrangement of subspaces $\mathcal A$: specifically, we must strengthen the hypothesis mentioned in the paragraph below Example 4.8. This hypothesis asked that no chosen configuration $A_i\in \mathcal{A}$ is \emph{equal} to the preimage of some set $A_i'$ under the coordinate projection maps. The following stronger assumption turns out to be sufficient for representation stability of $\mathcal{A}$-avoiding configuration spaces:

\begin{hypothesis}\label{hyp:no_axis}
We suppose that $\mathcal{A}$ is a finite collection of closed subsets $A_i\subseteq X^{S_i}$ that do \emph{not} contain any `coordinate axis': that is, identifying $X^{S_i}$ with $X^{S_i- \{s\}} \times X$, we must have
\[
 \{ \bar{x} \} \times X \not\subseteq A_i
\]
for any $\bar{x}\in X^{S_i- \{s\}}$ and any $s\in S_i$.
\end{hypothesis}

\begin{claim}\label{claim:order_ideal}
Under Hypothesis \ref{hyp:no_axis}, the multiplication
\begin{equation}
    P_{\mathcal{A}}(S) \times \{\hat{0}\} \subset P_{\mathcal{A}}(S)\times P_{\mathcal{A}}(1) \rightarrow P_{\mathcal{A}}(S\sqcup 1)
\end{equation}
identifies $P_{\mathcal{A}}(S)$ with an order ideal of $P_{\mathcal{A}}(S\sqcup 1)$.
\end{claim}

\begin{proof}Note that injectivity is obvious. Now suppose $B\in P_{\mathcal{A}}(S\sqcup 1)$ lies below $A\times \hat{0}$. We must show that $B$ is already in the image of the multiplication by $\hat{0}$, that is $B=B'\times \hat{0}$. Recall that the `bottom stratum' $\hat{0}\in P_{\mathcal{A}}(1)$ represents the entire space $X = X^1$. Thus the hypothesis is that $B\supseteq A\times X$ as subsets of $X^S\times X$.

Next, recall that the poset $P_{\mathcal{A}}(T)$ is formed by intersecting subspaces of the form 
\[
\left(\pi^T_{S_i}\right)^{-1} (A_i) \cong A_i \times X^{T\setminus j(S_i)} \text{ inside } X^{S_i} \times X^{T\setminus j(S_i)} \cong X^T
\]
for all injections $j:S_i\hookrightarrow T$. In particular, $B$ is the intersection of a collection of such subspaces where $T = S\sqcup 1$.

The claim would follow if we showed that all generators $\left(\pi^{T}_{S_i}\right)^{-1} (A_i) \supseteq B$ must come from injections $j: S_i \hookrightarrow S\sqcup 1$ that factor through some $j':S_i\hookrightarrow S$. Indeed, it would then follow that $B=B'\times X$ where $B'$ is the intersection of those $\left(\pi^S_{S_i}\right)^{-1} (A_i) \subseteq X^S$.

To see that the said factorization holds, for every such inclusion $j:S_i\hookrightarrow S\sqcup 1$ one gets an inclusion
\[
A\times X \subseteq B \subseteq \left(\pi^{T}_{S_i}\right)^{-1} (A_i)
\]
If it were the case that $1\in j(S_i)$ then $A_i$ would already contain a coordinate axis $\{\bar{a}\}\times X$, contradicting Hypothesis \ref{hyp:no_axis}. It follows that $j$ factors through $S$.
\end{proof}

% If we show that every generator $A_i\times X^{(S\coprod 1) \setminus j(S_i)} \hookleftarrow B$ is of the form $A_i\times X^{S\setminus j'(S_i)} \times X^{1}$, then it would follow that $B=B'\times X$ where $B'$ is the intersection of those $A_i\times X^{S\setminus j'(S_i)} \hookrightarrow X^S$. In other words, we only need to show that the injections $j:S_i\hookrightarrow S\coprod 1$ for which $B\hookrightarrow A_i\times X^{(S\coprod 1) \setminus j(S_i)}$ land inside $S$.

%Indeed, by assumption 
%\[
%\{\bar{x}\}\times X^1 \subseteq A\times X^1 \subseteq B \hookrightarrow A_i\times X^{(S\coprod 1) \setminus j(S_i)}
%\]
%for a point $\bar{x}\in A$. Since Hypothesis \ref{hyp:no_axis} ensures that $A_i$ may not contain a coordinate axis, it follows that the axis must come from $X^{(S\coprod 1) \setminus j(S_i)}$. But this axis appears in the coordinate $\{1\}\subseteq S\coprod 1$, which could happen only when $j(S_i)\subseteq S\coprod 1$ omits $\{1\}$. The claim follows.

Let us call an element $\beta\in P_{\mathcal{A}}(S)$ \emph{decomposable} if it is of the form $\beta'\times \hat{0}$ for some $\beta'\in P_{\mathcal{A}}(S')$, where $S' \subset S$ is a proper subset. If no such decompositions exist, call $\beta$ \emph{indecomposable}. Note that this is \emph{not} the same definition of indecomposability that is introduced just before Lemma 4.13 of \cite{spectralsequencestratification}.

With this definition of decomposability, the proof of Lemma 4.13 applies verbatim (in fact, the present definition of decomposability relates to that proof more naturally). Furthermore using the above Claim \ref{claim:order_ideal} to guarantee the order ideal assumption, the original proof of Theorem 4.15 using Lemma 4.14 is now valid as stated. To summarize, Theorem 4.15 remains valid if we assume in addition that the arrangement of subspaces $\mathcal A$ satisfies Hypothesis \ref{hyp:no_axis}. 

\section{Rectification} The next mistake which we hope to correct appears in the proof of representation stability with integer coefficients -- Subsection 4.6 of \cite{spectralsequencestratification}. The problem there is the repeated assertion that the chain level construction $S \mapsto R\Gamma^{-\bullet}(M^S,\mathbb D L^\bullet(S))$ is a twisted commutative algebra in chain complexes.  This is not true as stated since the functor $R \Gamma$ is not symmetric monoidal (it is only symmetric monoidal in the sense of `higher algebra', i.e.\ up to coherent homotopy); similar remarks apply to the functor $\mathbb D$. Thus one gets only what might be called a \emph{twisted $E_\infty$-algebra}, for which no good theory of representation stability has been developed (as of this time). In the proof of representation stability for cohomology with coefficients in a field this causes no problem since in this case we only need to work with the twisted commutative algebra given by the homology. But for the proof of stability with integer coefficients \cite[Subsection 4.6]{spectralsequencestratification} one needs to work on the chain level, due to the lack of a good K\"unneth isomorphism, and there the lack of strict commutativity means that representation stability theory does not apply. 

One can prove by an abstract rectification theorem that any twisted $E_\infty$-algebra is quasi-isomorphic to a genuine twisted commutative algebra, by arguments much like those in Proposition \ref{prop} below. This applies in particular to the functor $S \mapsto R\Gamma^{-\bullet}(M^S,\mathbb D L^\bullet(S))$. But this abstract rectification procedure will not necessarily preserve any finiteness properties, and the goal in Subsection 4.6 of the paper is precisely to construct a finitely generated FI-module from this twisted commutative algebra. Thus one needs to find a way to rectify and obtain a strict twisted commutative algebra without destroying the combinatorial structure of the algebra that goes into the proof. 

Let us make explicit what this combinatorial structure is. Let $\mathcal P$ be the category of pairs $(S,\beta)$ with $S$ a finite set and $\beta \in P_{\mathcal A}(S)$ a stratum; a morphism $(S,\beta) \to (S',\beta')$ is a bijection $f \colon S \to S'$ such that $\beta \geq f^\ast(\beta')$. Note that $P_{\mathcal A}(\varnothing) = \{\ast\}$. The category $\mathcal P$ is symmetric monoidal using $\sqcup$ on sets and the product maps $P_{\mathcal A}(S) \times P_{\mathcal A}(T) \hookrightarrow P_{\mathcal A}(S \sqcup T)$. The assignment 
$$ (S,\beta) \mapsto R\Gamma^{-\bullet}(\overline S_\beta,\mathbb D R) $$
is a lax symmetric monoidal $\infty$-functor $\mathcal P \to \mathrm{Ch}_R$, where $\overline S_\beta$ is the closure of the stratum inside $M^S$ corresponding to $\beta$. For every finite set $S$, the poset $P_{\mathcal A}(S)^{\mathrm{op}}$ is a subcategory of $\mathcal P$, and the constructed complex $R\Gamma^{-\bullet}(M^S,\mathbb D L^\bullet(S))$ is the total homotopy cofiber of the restricted diagram $P_{\mathcal A}(S)^{\mathrm{op}} \to \mathrm{Ch}_R$. Indeed, this follows from the fact that the construction $L^\bullet(S)$ can be identified with the fiber of a map to a homotopy limit, i.e.\ a total homotopy fiber \cite[Remark 3.2]{spectralsequencestratification}, and then $\mathbb DL^\bullet(S)$ is a total homotopy cofiber. What this means in more concrete terms is that $R\Gamma^{-\bullet}(M^S,\mathbb D L^\bullet(S))$ is the totalization of a double complex whose terms are of the form $R\Gamma^{-\bullet}(\overline S_\beta, \mathbb D R)$, where $\overline S_\beta$ is the closure of a stratum inside $M^S$. The differentials in the resulting complex are induced by the inclusions among the closures of strata. Moreover, the assignment $S \mapsto R\Gamma^{-\bullet}(M^S,\mathbb DL^\bullet(S))$ is a twisted $E_\infty$-algebra of chain complexes, whose product is induced by the K\"unneth maps
$$ R\Gamma^{-\bullet}(\overline S_\alpha, \mathbb D R) \otimes R\Gamma^{-\bullet}(\overline S_\beta, \mathbb D R) \to R\Gamma^{-\bullet}(\overline S_\alpha \times \overline S_\beta, \mathbb D R).$$

This product is the further combinatorial structure needed for the argument in Section 4 of \cite{spectralsequencestratification}. The construction in Section 3 of \cite{spectralsequencestratification} in fact shows that to any lax symmetric monoidal $\infty$-functor  $\mathcal P \to \mathrm{Ch}_R$ one can associate a twisted $E_\infty$-algebra $\mathrm{FB} \to \mathrm{Ch}_R$. Quasi-isomorphic functors produce quasi-isomorphic twisted $E_\infty$-algebras. And crucially, if the $\infty$-functor $\mathcal P \to \mathrm{Ch}_R$ is strict, i.e.\ a lax symmetric monoidal functor of $1$-categories, then the resulting twisted $E_\infty$-algebra is also strict, i.e.\ a twisted commutative algebra in the usual sense. A remark is that, as discussed above in Section \ref{subsec: product}, the monoidal product uses the Eilenberg--Zilber shuffle map; thus we are using that the Eilenberg--Zilber map is symmetric monoidal. 

%This is then the combinatorial structure which is needed for the argument: to any lax symmetric monoidal $\infty$-functor  $\mathcal P \to \mathrm{Ch}_R$ one can associate a twisted $E_\infty$-algebra $\mathcal{FB} \to \mathrm{Ch}_R$. Quasi-isomorphic functors produce quasi-isomorphic twisted $E_\infty$-algebras. If the $\infty$-functor $\mathcal P \to \mathrm{Ch}_R$ is strict, i.e.\ a lax symmetric monoidal functor of $1$-categories, then the resulting twisted $E_\infty$-algebra is also strict, i.e.\ a twisted commutative algebra in the usual sense. 
Then the finiteness results in \cite[Section 4]{spectralsequencestratification} are obtained from the assumption that the  complexes $R\Gamma^{-\bullet}(\overline S_\beta, \mathbb D R)$,  i.e.\ the values of the functor $\mathcal P \to \mathrm{Ch}_R$, have finitely generated cohomology that vanishes in a range.

% in the sense of higher category theory., i.e.\ an $E_\infty$-algebra in the functor category $[\mathcal P,\mathrm{Ch}_R]$ under Day convolution. 

The missing ingredient to getting a strict twisted commutative algebra is the following.
\begin{proposition}\label{prop} The $\infty$-functor $\mathcal P \to \mathrm{Ch}_R$ given by Borel--Moore chains
$$ \beta \mapsto R\Gamma^{-\bullet}(\overline S_\beta,\mathbb D R) $$
is quasi-isomorphic to a lax  symmetric monoidal functor of $1$-categories.
\end{proposition}

\begin{proof}
The structure of a lax symmetric monoidal $\infty$-functor $\mathcal P \to \mathrm{Ch}_R$ is equivalent to an $ E_\infty$-algebra in the functor category  $[\mathcal P,\mathrm{Ch}_R]$ under Day convolution. We will argue that every $E_\infty$-algebra in $[\mathcal P,\mathrm{Ch}_R]$ is weakly equivalent to a commutative monoid, which then corresponds to a lax symmetric monoidal functor of $1$-categories.

Let us consider the category of chain complexes $\mathrm{Ch}_R$ as a model category using the model structure for which the fibrations (weak equivalences) are the degreewise surjections (quasi-isomorphisms). We give the functor category $[\mathcal P,\mathrm{Ch}_R]$ the following slight variant of the projective model structure: a natural transformation $F \to G$ is a fibration (weak equivalence) if $F(S,\beta) \to G(S,\beta)$ is a fibration (weak equivalence) for all \emph{nonempty} $S$. Note that if $F \to G$ is a cofibration then $F(\varnothing,\ast) \to G(\varnothing,\ast)$ is an isomorphism.

We claim that this model category satisfies the \emph{commutative monoid axiom} of White \cite[Definition 3.1]{white}, which by \cite[Theorem 3.2]{white} implies the existence of an induced model structure on the category of commutative monoids in $[\mathcal P,\mathrm{Ch}_R]$. Secondly, the model category $[\mathcal P,\mathrm{Ch}_R]$ is \emph{symmetric flat} in the terminology of Pavlov--Scholbach \cite{ps}, which by their abstract rectification theorem \cite[Theorem 1.2]{ps} implies that the categories of $ E_\infty$-algebras and commutative monoids in  $[\mathcal P,\mathrm{Ch}_R]$ are Quillen equivalent, and finishes the proof.

Let $h \colon G \to F$ be a trivial cofibration. To verify the commutative monoid axiom we mush check that $h^{\Box n}/\mathbb S_n$ is a trivial cofibration as well, where $h^{\Box n}$ denotes the $n$-fold pushout-product of $h$. Note that 
$$ F^{\otimes n}(S,\beta)  = \!\! \bigoplus_{\substack{S = S_1 \sqcup \ldots \sqcup S_n \\ \beta = \beta_1 \times \ldots \times \beta_n}}\!\! F(S_1,\beta_1) \otimes \ldots \otimes F(S_n,\beta_n)$$
and that under this direct sum decomposition, $h^{\Box n}$ is a direct sum of maps of the form
$$ h(S_1,\beta_1)\, \Box\, h(S_2,\beta_2) \,\Box\, \ldots \,\Box\, h(S_n,\beta_n). $$
We will consider two types of summand. If one or more of the $S_i$ is empty, then one of the components of the pushout-product is an isomorphism, and then so is the whole pushout-product; clearly the corresponding components of $h^{\Box n}/\mathbb S_n$ are also isomorphisms and in particular trivial cofibrations. Restricting instead only to those summands of $F^{\otimes n}(S,\beta)$ such that none of the sets $S_i$ are empty, we may write
\begin{gather*}  \bigoplus_{\substack{S = S_1 \sqcup \ldots \sqcup S_n, \text{ all } S_i \neq \varnothing \\ \beta = \beta_1 \times \ldots \times \beta_n}}\!\! F(S_1,\beta_1) \otimes \ldots \otimes F(S_n,\beta_n) \\  \cong R[\mathbb S_n] \otimes \!\!\bigoplus_{\substack{S = S_1 \sqcup \ldots \sqcup S_n, \text{ all } S_i \neq \varnothing\\  \min(S_1) < \ldots < \min(S_n) \\ \beta = \beta_1 \times \ldots \times \beta_n}}\!\! F(S_1,\beta_1) \otimes \ldots \otimes F(S_n,\beta_n). \end{gather*}
But then the corresponding summands of $(F^{\otimes n}/\mathbb S_n)(S,\beta)$ is precisely 
$$ \bigoplus_{\substack{S = S_1 \sqcup \ldots \sqcup S_n, \text{ all } S_i \neq \varnothing\\  \min(S_1) < \ldots < \min(S_n) \\ \beta = \beta_1 \times \ldots \times \beta_n}}\!\! F(S_1,\beta_1) \otimes \ldots \otimes F(S_n,\beta_n),$$
and the corresponding component of $h^{\Box n}/\mathbb S_n$ is then the pushout-product $$ h(S_1,\beta_1)\, \Box\, h(S_2,\beta_2) \,\Box\, \ldots \,\Box\, h(S_n,\beta_n), $$ which is a trivial cofibration by the usual monoid axiom.

We should now verify symmetric flatness, i.e.\ that if $h$ is a cofibration, and $y$ is an equivariant weak equivalence\footnote{Meaning an equivariant map which is a weak equivalence on the underlying objects.} between two objects with $\mathbb S_n$-action, then $y \Box_{\mathbb S_n} h^{\Box n}$ is a weak equivalence. The argument is extremely similar to the verification of the commutative monoid axiom, and uses again that an $n$-fold pushout-product $h^{\Box n}$ of a cofibration $h$ decomposes as a direct sum of isomorphisms and summands on which $\mathbb S_n$ acts freely. Thus a similar argument shows that $y \Box_{\mathbb S_n} h^{\Box n}$ is pointwise given by a direct sum of pushout-products of the form $y_1 \Box h_1 \Box \ldots \Box h_n$ where $y_1$ is a weak equivalence and each $h_i$ is a cofibration in $\mathrm{Ch}_R$; the fact that such a pushout-product is a weak equivalence says precisely that the model category $\mathrm{Ch}_R$ is \emph{flat} \cite[Definition 2.1(vi)]{ps}, which is easy to show.  
\end{proof}

Note that $[\mathcal P,\mathrm{Ch}_R]$ does not have cofibrant unit, which is why we could not directly quote the rectification theorem for $ E_\infty$-algebras proven by White and we instead use the result of Pavlov--Scholbach.  

Introducing the category $\mathcal P$ and formulating the arguments in terms of monoidal functors out of $\mathcal P$ also clarifies the subsequent discussion surrounding Lemma 4.18 in \cite{spectralsequencestratification}. We may formulate a more general version of Lemma 4.18 in terms of the category $\mathcal P$ as follows:

\begin{proposition}Let $F \colon \mathcal P \to \mathrm{Ch}_R$ be a lax symmetric monoidal functor with the following properties:
\begin{enumerate}
    \item The product maps $F(S,\beta) \otimes F(1,\hat 0) \to F(S\sqcup 1,\beta \times \hat 0)$ are quasi-isomorphisms.\footnote{Note that in accordance with the previous part of this corrigendum, we are now using the modified definition of what it means for a stratum to be decomposable.} 
    \item The homology $ H_\ast(F(S,\beta))$ is finitely generated for all $(S,\beta)$. If $\sigma(S,\beta)=p$ then $H_k(F(S,\beta))=0$ for $k>d \vert S \vert - 2p$, and $H_{d\vert S \vert-2p}(F(S,\beta))$ is a free $R$-module.
    \item $H_d(F(1,\hat 0))\cong R$, where $\hat 0 \in P_{\mathcal A}(1)$ denotes the bottom element. 
\end{enumerate}
Then $F$ is quasi-isomorphic to a functor $C$ such that:
\begin{enumerate}
    \item The product maps $C(S,\beta) \otimes C(1,\hat 0) \to C(S \sqcup 1,\beta \times \hat 0)$ are isomorphisms.
    \item $C(S,\beta)$ is a bounded complex of finitely generated free modules, vanishing above degree $d \vert S \vert - 2 \sigma (S,\beta)$.
    \item $C_d(1,\hat 0) \cong R$. 
\end{enumerate}
\end{proposition}

The proof is the same as that of Lemma 4.18 of \cite{spectralsequencestratification}. This formulation also corrects an oversight in the formulation of Lemma 4.18: there the system of quasi-isomorphisms is not explicitly assumed to be compatible with the symmetric group actions, which is necessary in order that the construction produces a twisted commutative algebra. Here the symmetric group actions are part of the structure of the category $\mathcal P$.

\printbibliography

\end{document}